\documentclass{amsart}

\newtheorem{theorem}{Theorem}[section]
\newtheorem{lemma}[theorem]{Lemma}
\newtheorem{corollary}[theorem]{Corollary}
\newtheorem{proposition}[theorem]{Proposition}
\theoremstyle{definition}
\newtheorem{definition}[theorem]{Definition}
\newenvironment{example}[1][Example]{\begin{trivlist}
\item[\hskip \labelsep {\bfseries #1}]}{\end{trivlist}}

\newenvironment{remark}[1][Remark]{\begin{trivlist}
\item[\hskip \labelsep {\bfseries #1}]}{\end{trivlist}}

\numberwithin{equation}{section}

\newcommand{\abs}[1]{\lvert#1\rvert}

\usepackage{cite}
\usepackage{units}
\usepackage[scr=rsfso]{mathalfa}
\newcommand{\LambertL}{\mathscr{L}}

\begin{document}

\title{Lambert series and q-functions near q=1}

\author{Shubho Banerjee}
\address{Department of Physics, Rhodes College, 2000 N. Parkway, Memphis, TN 38112}
\email{banerjees@rhodes.edu}
\email{wilbe-16@rhodes.edu}
\thanks{The first author was supported by the Van Vleet Physics Professorship.}

\author{Blake Wilkerson}
\thanks{The second author was supported by the Mac Armour Physics Fellowship.}

\date{January 2016}

\begin{abstract}
We study the Lambert series $\mathscr{L}_q(s,x) = \sum_{k=1}^\infty k^s q^{k x}/(1-q^k)$, for all $s \in \mathbb{C}$. We obtain the complete asymptotic expansion of $\mathscr{L}_q(s,x)$ near $q=1$. Our analysis of the Lambert series yields the asymptotic forms for several related q-functions: the q-gamma and q-polygamma functions, the q-Pochhammer symbol, and, in closed form, the Jacobi theta functions. Some typical results include $\Gamma_2(\frac{1}{4}) \Gamma_2(\frac{3}{4}) \simeq \frac{2^{13/32} \pi}{\log 2}$ and $\vartheta_4 (0,e^{-1/\pi}) \simeq 2 \pi e^{-\pi^3\!/4}$, with relative errors of order $10^{-25}$ and $10^{-27}$ respectively.
\end{abstract}

\maketitle

\section{Introduction}\label{Introduction}
We analyze the $q \rightarrow 1^-$ asymptotic behavior of a Lambert series and its associated functions: the q-Pochhammer symbol, the q-gamma and q-polygamma functions, and the Jacobi theta functions. This section provides the definitions for all the functions analyzed in this paper and outlines the relationships that connect these functions to each other.

\subsection*{Lambert series}
We study Lambert series of the type
\begin{equation} \label{Lambert}
\hspace{1.02 in} \LambertL_q(s,x) = \sum_{k=1}^\infty \frac{k^s q^{k x}}{1-q^k},  \hspace{0.5 in} s \in \mathbb{C},
\end{equation}
with $0 \le q < 1$ and $x>0$. This Lambert series is closely related to the polylogarithm function
\begin{equation} \label{polylog}
\hspace{1.37 in} \text{Li}_s(z) = \sum_{k=1}^\infty \frac{z^k}{k^s} \qquad\qquad s \in \mathbb{C}, \, \abs{z}<1
\end{equation}
through the equation
\begin{equation*}
\sum_{n=0}^\infty \text{Li}_s(q^{n+x}) = \LambertL_q(-s,x).
\end{equation*}
At $x=1$, the Lambert series $\LambertL_q(s,x)$ is the generating function for the divisor function $\sigma_s(n)$, the sum of the $s^\text{th}$ powers of divisors of an integer $n$ \cite{Abramowitz}:
\begin{equation}\label{divisor}
\sum_{n=1}^\infty \sigma_s(n) \, q^n = \LambertL_q(s,1).
\end{equation}

\subsection*{q-Pochhammer symbol}
The q-Pochhammer symbol is defined as the infinite product
\begin{equation}\label{q-Pochhammer}
(a,q)_\infty = \prod_{n=0}^\infty (1-a q^n).
\end{equation}
By setting $a=q^x$, the q-Pochhammer symbol can be related to the Lambert series at $s=-1$ as (see Lemma \ref{LemmaPochhammer})
\begin{equation}\label{log q-Pochhammer}
\log\,(q^x,q)_\infty = -\LambertL_q(-1,x).
\end{equation}

\subsection*{q-gamma and q-polygamma functions}
Jackson\cite{Jackson} introduced the following q-analog of the gamma function, $\Gamma(x)$:
\begin{gather} \label{qgamma}
\Gamma_q(x) = (1-q)^{1-x} \prod_{n=0}^\infty \frac{1-q^{n+1}}{1-q^{n+x}} \\
\hspace{0.12 in} = (1-q)^{1-x} \frac{(q,q)_\infty}{(q^x,q)_\infty} \nonumber.
\end{gather}
The q-gamma function, stated in terms of the ratio of two q-Pochhammer symbols, is thus related to the Lambert series.
The q-digamma function, $\psi_q(x)= d\log\Gamma_q(x)/dx $, can be expressed in terms of the Lambert series (\ref{Lambert}) at $s=0$ as
\begin{gather}
\psi_q(x) = -\log(1-q) + \log q \sum_{n=0}^\infty \frac{q^{n+x}}{1-q^{n+x}} \label{qdigamma} \\
\hspace{0.22 in} = -\log(1-q) + \log q \sum_{k=1}^\infty \frac{q^{k x}}{1-q^k} \nonumber \\
\hspace{0.12 in} = -\log(1-q) + \log q \, \LambertL_q(0,x) \nonumber.
\end{gather}
The higher order q-polygamma functions, $\psi^{(m)}_q(x) = d^m\psi_q(x)/dx^m $, for integers $m>0$, can be expressed in terms of the Lambert series at $s=m$ as
\begin{equation}\label{qpolygamma}
\psi^{(m)}_q(x) = (\log q)^{m+1} \LambertL_q(m,x).
\end{equation}

\subsection*{Jacobi theta functions}
The Jacobi theta functions, defined for $z \in \mathbb{C}$ as
\begin{equation}\label{theta 1}
\vartheta_1\!\left(z+\frac{\pi}{2},q\right) = \vartheta_2(z,q) = \sum_{n=-\infty}^\infty q^{(n+1/2)^2} e^{i (2 n+1) z}
\end{equation}
and
\begin{equation}\label{theta 3}
\vartheta_3\left(z,q\right) = \vartheta_4\!\left(z+\frac{\pi}{2},q\right) = \sum_{n=-\infty}^\infty q^{n^2} e^{i 2 n z},
\end{equation}
play an important role in the theory of elliptic functions. The theta functions can be related to the Lambert series through the q-Pochhammer symbol. Using the Jacobi triple product identity \cite{Whittaker},
\begin{equation}\label{triple product}
\vartheta_4(z,q) = \left(q^2,q^2\right)_\infty \, \left(q \, e^{i 2 z},q^2\right)_\infty \, \left(q \, e^{-i 2 z},q^2\right)_\infty \, .
\end{equation}
The other three theta functions defined above can be written in terms of $\vartheta_4(z,q)$ using (\ref{theta 1}), (\ref{theta 3}) and
\begin{equation}\label{theta1-theta4}
\vartheta_1(z,q) = -i \, q^{1/4} \, e^{i z} \, \vartheta_4\!\left(\! z+\frac{\log q}{2 i},q\right).
\end{equation}

The Lambert series (\ref{Lambert}) has a singularity at $q=1$ that makes the analysis of these series difficult in this limit. As $q \rightarrow 1^-$, the number of terms of $\LambertL_q(s,x)$ required to achieve a given accuracy increases dramatically. The associated functions described above have similar convergence problems in this limit.

In Section \ref{Lambert result} below, we obtain an asymptotic expansion for the Lambert series near $q=1$. This expansion converges much faster near $q=1$ than the series in (\ref{Lambert}). Only two terms in this expansion give accurate results within one part in ten million for $0.1<q<1$. In the subsequent sections we obtain similar asymptotic expansions for the associated functions discussed above. Some well-known asymptotic results (such as that for the Lambert series by Knopp \cite{Knopp}, and for the Euler function by Watson \cite{Watson}) follow as special cases of the theorems proved in this paper.

\section{Lambert series near $q=1$}
\label{Lambert result}
We begin by presenting some preliminary definitions and properties that will be utilized throughout the paper. Recall that the Bernoulli polynomials $B_n(x)$ are defined through the exponential generating function
\begin{equation} \label{B(x)}
\frac{t \, e^{x t}}{e^t-1}=\sum_{n=0}^\infty \frac{B_n(x)}{n!} \, t^n.
\end{equation}
Setting $x=0$ in (\ref{B(x)}) gives the generating function for Bernoulli numbers $B_n = B_n(0)$. Directly related to the Bernoulli generating functions is the differential operator $\frac{D}{e^D-1}$, where $D=\frac{d}{dx}$ represents differentiation with respect to $x$. By formally expanding the operator as a power series in $D$,
\begin{equation} \label{D}
\frac{D}{e^D-1}=\sum_{n=0}^\infty \frac{B_n}{n!} \, D^n,
\end{equation}
it can be applied to a function of $x$.
For the proofs below, we require the following elementary applications of this operator \cite{Hasse}:
\begin{gather}
\hspace{0.65 in} \frac{D}{e^D-1} \, x^m = B_m(x), \qquad m \in \mathbb{N} \label{x^m} \\
\hspace{0.63 in} \frac{D}{e^D-1} \, x^{-s} = s \, \zeta(1+s,x), \qquad s \in \mathbb{C} \label{x^s}
\end{gather}
where $\zeta(s,x)$ is the Hurwitz zeta function, a generalization of the Riemann zeta function $\zeta(s) = \zeta(s,1)$.

Having introduced the definitions and properties above, we now present the main results of this section, beginning with the following lemma.

\begin{lemma} \label{LemmaLambert}
The Lambert series \emph{(\ref{Lambert})} can be expressed as the application of the differential operator \emph{(\ref{D})} on the polylogarithm function in the following manner\emph{:}
\begin{equation*}
\LambertL_q(s,x) = \frac{D}{e^D-1} \frac{\emph{Li}_{1-s}(q^x)}{\log(1/q)}.
\end{equation*}
\end{lemma}

\begin{proof}
\begin{equation*}
\begin{aligned}
\frac{D}{e^D-1} \, \text{Li}_{1-s}(q^x) &= \frac{D}{e^D-1} \sum_{k=1}^\infty \frac{q^{k x}}{k^{1-s}} \\
&= \sum_{k=1}^\infty k^{s-1} \frac{D}{e^D-1} \, q^{k x} \\
&= \sum_{k=1}^\infty k^{s-1} \sum_{n=0}^\infty \frac{B_n}{n!} \, D^n \, q^{k x} \\
&= \sum_{k=1}^\infty k^{s-1} \sum_{n=0}^\infty \frac{B_n}{n!} \, (k \log q)^n \, q^{k x} \\
&= \sum_{k=1}^\infty k^{s-1} \frac{k \log q}{q^k-1} \, q^{k x} \\
&= \log(1/q) \sum_{k=1}^\infty \frac{k^s q^{k x}}{1-q^k} \\
&= \log(1/q) \, \LambertL_q(s,x).
\end{aligned}
\end{equation*}
Dividing both sides by $\log(1/q)$ proves the lemma.
\end{proof}

Lemma \ref{LemmaLambert} establishes an important relationship between the Lambert series (\ref{Lambert}) and the polylogarithm functions. We now use this relationship to prove one of the main results of this paper as stated in the theorem below.

\begin{theorem}\label{TheoremLambert}
The Lambert series $\LambertL_q(s,x)$ has the following expansion at $q=1$.
\begin{enumerate}
\item For $s \neq 0, -1, -2, \ldots$ \emph{:}

\begin{equation*}
\begin{aligned}
\LambertL_q(s,x) = \frac{\Gamma(1+s) \, \zeta(1+s,x)}{\left(\!\log\frac{1}{q}\right)^{\hspace{-0.035 in}1+s}} - \sum_{k=0}^\infty \frac{\zeta(1-s-k)}{k!} B_k(x) (\log q)^{k-1}.
\end{aligned}
\end{equation*}

\item For $s = 0$\emph{:}
\begin{equation*}
\begin{aligned}
\hspace{-0.43 in} \LambertL_q(0,x) = \frac{\psi(x)+ \log\log\frac{1}{q}}{\log q} - \sum_{k=1}^\infty \frac{\zeta(1-k)}{k!} B_k(x) (\log q)^{k-1}.
\end{aligned}
\end{equation*}

\item For $s = -m = -1, -2, -3, \ldots$ \emph{:}
\begin{equation*}
\begin{aligned}
\hspace{0.32 in} \LambertL_q(-m,x) = \bigg[m \, \zeta^{\,\prime}(1-m,x) &+ \bigg(\!\log\log\frac{1}{q} - H_{m-1}\!\bigg) B_m(x)\bigg] \frac{(\log q)^{m-1}}{m!} \\
&- \sum_{\substack{k=0, \\ k \neq m}}^\infty \frac{\zeta(1+m-k)}{k!} B_k(x) (\log q)^{k-1},
\end{aligned}
\end{equation*}

\end{enumerate}
\hspace{0.25 in} where $H_{m-1}$ is the $(m-1)^\text{th}$ harmonic number.
\end{theorem}

\begin{proof}
We use the series expansion of the polylogarithm $\text{Li}_{1-s}(z)$ about $z=1$, given in \cite{EMOT}. To prove (1), we use the expansion of $\text{Li}_{1-s}(z)$ valid for $s \neq 0, -1, -2, \ldots$, and set $z = q^x$. Then applying the operator (\ref{D}) to both sides gives:
\begin{equation*}
\begin{aligned}
\frac{D}{e^D-1} \text{Li}_{1-s}(q^x) &= \frac{D}{e^D-1} \left[\Gamma(s) \! \left(\! \log \frac{1}{q^{x}} \!\right)^{\hspace{-0.035 in}-s} + \sum_{k=0}^\infty \frac{\zeta(1-s-k)}{k!} (\log q^x)^k\right] \\
&= \Gamma(s) \, \frac{D}{e^D-1} x^{-s} \! \left(\! \log\frac{1}{q}\right)^{\hspace{-0.035 in}-s} + \sum_{k=0}^\infty \frac{\zeta(1-s-k)}{k!} \, \frac{D}{e^D-1} x^k \, (\log q)^k \\
&= \Gamma(1+s) \, \zeta(1+s,x) \! \left(\!\log\frac{1}{q}\right)^{\hspace{-0.035 in}-s} + \sum_{k=0}^\infty \frac{\zeta(1-s-k)}{k!} B_k(x) (\log q)^k.
\end{aligned}
\end{equation*}

To establish the result above we invoked (\ref{x^m}), (\ref{x^s}), and the recursion relation $\Gamma(1+s) = s \, \Gamma(s)$. Dividing both sides by $\log(1/q)$ and applying Lemma \ref{LemmaLambert} completes the proof of (1).

To prove (2) and (3), we use the series expansion of $\text{Li}_{1+m}(z)$ that holds for $m = 0, 1, 2, \ldots$, and set $z = q^x$. Applying the operator (\ref{D}) to both sides gives:
\begin{equation*}
\begin{aligned}
\frac{D}{e^D-1} \text{Li}_{1+m}&(q^x) \\
&\text{\hspace{-0.63 in}}= \frac{D}{e^D-1} \Bigg[\!\!\left(\!H_m - \log \log \frac{1}{q^{x}}\!\right) \frac{(\log q^x)^m}{m!} + \sum_{\substack{k=0, \\ k \neq m}}^\infty \frac{\zeta(1+m-k)}{k!} (\log q^x)^k \Bigg] \\
&\text{\hspace{-0.63 in}}= -\left[\frac{D}{e^D-1} (x^m \log x) + \left(\!\log\log\frac{1}{q} - H_m\!\right) \frac{D}{e^D-1} x^m \right] \frac{(\log q)^m}{m!} \\
&\text{\hspace{1.33 in}}+ \sum_{\substack{k=0, \\ k \neq m}}^\infty \frac{\zeta(1+m-k)}{k!} \, \frac{D}{e^D-1} x^k \, (\log q)^k \\
&\text{\hspace{-0.63 in}}= -\left[m \, \zeta^{\,\prime}(1-m,x) + \left(\!\log\log\frac{1}{q} - H_{m-1}\!\right) B_m(x)\right] \frac{(\log q)^m}{m!} \\
&\text{\hspace{1.33 in}}+ \sum_{\substack{k=0, \\ k \neq m}}^\infty \frac{\zeta(1+m-k)}{k!} B_k(x) (\log q)^k .
\end{aligned}
\end{equation*}

The last step above assumes $m \ne 0$. Dividing both sides by $\log(1/q)$ and applying Lemma \ref{LemmaLambert} completes the proof of (3).

The derivative of the zeta function $\zeta^{\,\prime}(1-m,x)$ in the proof of (3) is obtained in the following manner:
\begin{equation*}
\begin{aligned}
\frac{D}{e^D-1} (x^m \log x) &= \frac{D}{e^D-1} \left(\frac{d}{ds} \, x^s \!\right)_{\!\!s=m} \\
&= \frac{d}{ds} [-s \, \zeta(1-s,x)]_{s=m} \\
&= m \, \zeta^{\,\prime}(1-m,x) - \zeta(1-m,x) \\
&= m \, \zeta^{\,\prime}(1-m,x) + \frac{1}{m} B_m(x).
\end{aligned}
\end{equation*}
The coefficients of $B_m(x)$ were combined to give $H_m - (1/m) = H_{m-1}$ in the final result. For the proof of (2), i.e. the $m=0$ case, we get $\frac{D}{e^D-1} \log(x) = \psi(x)$. Otherwise, the proof follows the same steps as that of (3).
\end{proof}

\begin{remark}
The most natural interval for $x$ in the asymptotic expansions of Theorem \ref{TheoremLambert} is $0 < x \le 1$. For $x$ outside of this interval, we may use the Lambert series recursion relation
\begin{equation}
\LambertL_q(s,x) = \LambertL_q(s,x + 1 - \lceil x \rceil) - \! \sum_{n=1}^{\lceil x \rceil - 1} \text{Li}_{-s}\big(q^{n + x - \lceil x \rceil}\big)
\end{equation}
to apply Theorem \ref{TheoremLambert} to $\LambertL_q(s,x + 1 - \lceil x \rceil)$ instead of $\LambertL_q(s,x)$. Here $\lceil x \rceil$ denotes the ceiling of $x$, so that $0 < x + 1 - \lceil x \rceil \le 1$. For using the Theorems in Sections \ref{SectionPochhammer}, \ref{SectionGamma}, \ref{SectionDigamma}, and \ref{SectionPolygamma}, the corresponding $x$ to $x + 1 - \lceil x \rceil$ recursion should be applied when $x>1$.
\end{remark}

Note that at $x=1$ Theorem \ref{TheoremLambert} is in agreement with Knopp's asymptotic result \cite{Knopp} for $Re\,s>0$:
\begin{equation*}
\LambertL_q(s,1) \sim \frac{\Gamma(1+s)\zeta(1+s)}{(1-q)^{1+s}}, ~~~~~~\text{as}~~q\rightarrow 1.
\end{equation*}
In addition, application of Theorem \ref{TheoremLambert} at $x=1$ yields the complete asymptotic expansion of the generating function for the divisor function (\ref{divisor}) for any $s$. Of particular interest are the asymptotic expansions when $s$ is an odd integer for which the asymptotic form converges to a finite series as stated in the corollary below.

\begin{corollary}\label{CorollaryLambert}
At odd integer values of $s$, the generating function for the divisor function $\sigma_s(n)$ \emph{(\ref{divisor})} has the following asymptotic expansion at $q=1$.
\begin{enumerate}
\item For $s = m = 1, 3, 5, \ldots$ \emph{:}
\begin{equation*}
\sum_{n=1}^\infty \sigma_m(n) \, q^n \simeq \frac{m! \, \zeta(1+m)}{(\log q)^{1+m}}-\frac{\zeta(1-m)}{\log q}-\frac{1}{2}\zeta(-m).
\end{equation*}

\item For $s = -m = -1, -3, -5,\ldots$ \emph{:}
\begin{equation*}
\begin{aligned}
\hspace{0.4 in} \sum_{n=1}^\infty \sigma_{-m}(n)\, q^n \simeq \bigg[m \, \zeta^{\,\prime}(1-m) &+ \left(\!\log\log\frac{1}{q} - H_{m-1}\!\right) B_m(1)\bigg] \frac{(\log q)^{m-1}}{m!} \\
&- \sum_{\substack{k=0, \\ k \neq m}}^{m+1}  \frac{\zeta(1+m-k)}{k!} B_k(1)\,(\log q)^{k-1}.
\end{aligned}
\end{equation*}
\end{enumerate}
The symbol $\simeq$ indicates asymptotic equality in the limit $q\rightarrow1^{-}$.
\end{corollary}
\begin{proof}
The proof is a straightforward application of Theorem \ref{TheoremLambert} while using (\ref{divisor}). The terms with higher $k$ values [$k\ge 2$ in (1) and $k \ge m+2$ in (2)] do not contribute to the sum because the product of the Riemann zeta function with $B_k(1)$ for these values of $k$ is always zero for odd $m$.
\end{proof}
\begin{remark}
The final asymptotic forms in Corollary \ref{CorollaryLambert} are obtained through the cancellation of infinite series in Theorem \ref{TheoremLambert} that diverge for general $x$ values. At $x=1$, the otherwise diverging coefficients are multiplied by zeros to yield finite asymptotic expressions. These expressions are not valid in the limit $q\rightarrow 0^+$ but are asymptotically equal in the limit $q\rightarrow 1^-$. We replace the $=$ signs with the $\simeq$ signs to indicate this fact. Same notation of the $\simeq$ sign is used throughout this paper wherever cancellation of infinite diverging series takes place due to a special choice of $x$ values.
\end{remark}

For the positive odd values of $m$ in (1) of Corollary \ref{CorollaryLambert}, the generating function of the divisor function (\ref{divisor}) is intimately connected to the Fourier expansion of the Eisenstein series\cite{Apostol2} through the relation
\begin{equation}\label{Eisenstein}
E_{2k}(q) = 1 + \frac{2}{\zeta(1-2k)} \sum_{n=1}^\infty \sigma_{2k-1}(n) \, q^n.
\end{equation}

Note that the second term of the asymptotic expansion in (1) of Corollary \ref{CorollaryLambert} vanishes for all $m$ except $m=1$. We redefine the Eisenstein series in the following manner to account for this exception.
\begin{definition}\label{DefinitionEisenstein}
For $k \ge 1$ let the \textit{modified} Eisenstein series be
\begin{equation*}
{\tilde E}_{2k}(q) = E_{2k}(q) + \frac{2\,\zeta(2-2k)}{\zeta(1-2k)\log q}.
\end{equation*}
\end{definition}
With this definition ${\tilde E}_{2k} = {E}_{2k}$ for $k \ge 2$ but ${\tilde E}_{2}(q) = {E}_{2}(q) + 12/\log q$. We may now state the theorem below.
\begin{theorem} \label{TheoremEisenstein}
For $k\ge 1$ the modified Eisenstein series have the following asymptotic form at $q=1$\emph{:}
\begin{equation*}
{\tilde E}_{2k}(q) \simeq \left(\frac{2\pi i}{\log q}\right)^{\!{2k}}.
\end{equation*}
\end{theorem}
\begin{proof}
Rewriting the result in Corollary \ref{CorollaryLambert} in terms of ${\tilde E}_{2k}$ we get
\begin{equation*}
\begin{aligned}
{\tilde E}_{2k}(q) &\simeq \frac{(2k-1)!\,2\,\zeta(2k)}{\zeta(1-2k)(\log q)^{2k}} \\
&= \frac{\,2\,(2k)!\,\zeta(2k)}{2k \,\zeta(1-2k)\,(\log q)^{2k}} \\
&= \frac{(-1)^{k+1}(2\pi)^{2k}B_{2k}}{-B_{2k}\,(\log q)^{2k}} \\
&= \left(\frac{2\pi i}{\log q}\right)^{\!\!{2k}}.
\end{aligned}
 \end{equation*}
\end{proof}
\begin{remark}
The asymptotic expression in Theorem \ref{TheoremEisenstein} has no additional terms. The expression is extremely accurate even when $q \not\approx 1$. For example, at $ q = 0.1$, $0.3$, and $0.5$, the relative errors are less than or equal to order $10^{-5}$, $10^{-12}$, and  $10^{-15}$ respectively for all $k\ge 1$.
\end{remark}

From Theorem \ref{TheoremEisenstein} this corollary about the product of two Eisenstein series follows naturally.
\begin{corollary} \label{CorollaryEisenstein}
For any two integers $k,l\ge 1$ the modified Eisenstein series satisfy the asymptotic relation\emph{:}
\begin{equation*}
{\tilde E}_{2k}(q) \,{\tilde E}_{2l}(q) \simeq  {\tilde E}_{2k+2l}(q).
\end{equation*}
\end{corollary}
\begin{proof}
The proof is trivial given the asymptotic form of the Eisenstein series in Theorem \ref{TheoremEisenstein}.
\end{proof}
The Eisenstein series $E_{2k}(q)$ for all $k\ge 4$ can be expressed as polynomials in $E_4(q)$ and $E_6(q)$. The asymptotic forms of the modified Eisenstein series satisfy these polynomial identities. Additionally, the modified series satisfy a much simpler asymptotic relation  $\tilde {E}_{2k}(q)\simeq[\tilde {E}_{2}(q)]^k$.

\section{q-Pochhammer symbol}\label{SectionPochhammer}
To analyze the behavior of the q-Pochhammer symbol near $q=1$ we first relate the infinite product to the Lambert series through the following lemma.

\begin{lemma} \label{LemmaPochhammer}
The q-Pochhammer symbol \emph{(\ref{q-Pochhammer})} at $a=q^x$ can be expressed in terms of the Lambert series \emph{(\ref{Lambert})} at $s=-1$ as\emph{:}
\begin{equation*}
(q^x,q)_\infty = e^{-\LambertL_q(-1,x)}.
\end{equation*}
\end{lemma}

\begin{proof}
\begin{equation*}
\begin{aligned}
\log\,(q^x,q)_\infty &= \log \prod_{n=0}^\infty (1- q^{n+x}) = \sum_{n=0}^\infty \log (1- q^{n+x}) \\
&= -\sum_{n=0}^\infty \sum_{k=1}^\infty \frac{q^{k(n+x)}}{k} = -\sum_{k=1}^\infty \frac{q^{k x}}{k}\sum_{n=0}^\infty q^{k n} \\
&= -\sum_{k=1}^\infty \frac{1}{k} \frac{q^{k x}}{1-q^k} = -\LambertL_q(-1,x).
\end{aligned}
\end{equation*}
Exponentiating both sides completes the proof of the lemma.
\end{proof}

The main result of this section describing the asymptotic behavior of the q-Pochhammer symbol may now be stated as this theorem below.
\begin{theorem}\label{TheoremPochhammer}
The expansion of the q-Pochhammer symbol \emph{(\ref{q-Pochhammer})} at $q=1$ is\emph{:}
\begin{equation*}
(q^x,q)_\infty = \frac{\sqrt{2 \pi}}{\Gamma(x)} \left(\!\log \frac{1}{q}\right)^{\!\!\frac{1}{2}-x} \prod_{\substack{k=0, \\ k \neq 1}}^\infty e^{\frac{\zeta(2-k)}{k!} B_k(x) (\log q)^{k-1}}.
\end{equation*}
\end{theorem}

\begin{proof}
\begin{equation*}
\begin{aligned}
\LambertL_q(-1,x) &= \zeta^{\,\prime}(0,x) + B_1(x) \left(\!\log\log\frac{1}{q} - H_0\!\right) - \sum_{\substack{k=0, \\ k \neq 1}}^\infty \frac{\zeta(2-k)}{k!} B_k(x) (\log q)^{k-1} \\
&= \log \frac{\Gamma(x)}{\sqrt{2 \pi}} + \left(\!x-\frac{1}{2}\right) \log\log\frac{1}{q} - \sum_{\substack{k=0, \\ k \neq 1}}^\infty \frac{\zeta(2-k)}{k!} B_k(x) (\log q)^{k-1}.
\end{aligned}
\end{equation*}
Substituting the final expression for $\LambertL_q(-1,x)$ into Lemma \ref{LemmaPochhammer} completes the proof.
\end{proof}

\begin{remark}
The asymptotic form of the Euler function $(q,q)_\infty$ follows as a special case of this theorem at $x=1$:
\begin{equation} \label{Euler}
(q,q)_\infty \simeq \sqrt{\frac{2 \pi}{\log(1/q)}} \, e^\frac{\pi^2}{6 \log q} q^{-\frac{1}{24}}.
\end{equation}
It agrees with Watson's result for this function \cite{Watson}.
\end{remark}

We exploit the symmetries of the Bernoulli polynomials about $x=1/2$ in the following corollary of Theorem \ref{TheoremPochhammer}.
\begin{corollary} \label{CorollaryPochhammer}
The q-Pochhammer symbol \emph{(\ref{q-Pochhammer})} satisfies the asymptotic reflection formula\emph{:}
\begin{equation*}
\left(q^x,q\right)_\infty \left(q^{1-x},q\right)_\infty \simeq 2 \sin(\pi x)\, e^\frac{\pi^2}{3 \log q} q^{-\frac{1}{2} \left( \frac{1}{6} - x + x^2 \right)}.
\end{equation*}
\end{corollary}

\begin{proof}
\begin{equation}
\begin{aligned}
\left(q^x,q\right)_\infty \left(q^{1-x},q\right)_\infty &= \frac{\sqrt{2 \pi}}{\Gamma(x)} \left(\!\log \frac{1}{q}\right)^{\!\!\frac{1}{2}-x} \prod_{\substack{k=0, \\ k \neq 1}}^\infty e^{\frac{\zeta(2-k)}{k!} B_k(x) (\log q)^{k-1}} \\
&\hspace{0.16 in} \times \frac{\sqrt{2 \pi}}{\Gamma(1-x)} \left(\!\log \frac{1}{q}\right)^{\!x-\frac{1}{2}} \prod_{\substack{k=0, \\ k \neq 1}}^\infty e^{\frac{\zeta(2-k)}{k!} B_k(1-x) (\log q)^{k-1}} \\
&= \frac{2 \pi}{\Gamma(x) \Gamma(1-x)} \prod_{\substack{k=0, \\ k \neq 1}}^\infty e^{\frac{\zeta(2-k)}{k!} [B_k(x) + B_k(1-x)] (\log q)^{k-1}} \\
&= 2 \sin(\pi x) \, e^{\frac{\zeta(2)}{\log q} [B_0(x) + B_0(1-x)]} \, e^{\frac{\zeta(0)}{2} [B_2(x) + B_2(1-x)] \log q} \\
&\hspace{0.16 in} \times \prod_{k=3}^\infty e^{\frac{\zeta(2-k)}{k!} [B_k(x) + B_k(1-x)] (\log q)^{k-1}} \\
&\simeq 2 \sin(\pi x) \, e^{\frac{\pi^2}{3 \log q}} q^{-\frac{1}{2} \left(\frac{1}{6}-x+x^2\right)},
\end{aligned}
\end{equation}
where we used the reflection formula for the gamma function \cite{Abramowitz}. In the proof, the product from $k=3$ to $\infty$ equates to $1$ because $\zeta(2-k)$ is zero for even $k>2$, and $B_k(x) + B_k(1-x) = 0$ for odd $k$.
\end{proof}

To highlight the accuracy of Corollary \ref{CorollaryPochhammer} when $q\not\approx 1$, we provide the following example.
\begin{example} \label{PropositionPochhammer}
\emph{The q-Pochhammer reflection formula at $x=1/4$ and $x=3/4$ is given by the asymptotic relation}:
\begin{equation*}
\big(q^{1/4},q\big)_\infty \big(q^{3/4},q\big)_\infty \simeq \sqrt{2} \, e^{\frac{\pi^2}{3 \log q}} q^\frac{1}{96}.
\end{equation*}
\end{example}
At $q=0.001$, $0.01$, and $0.1$, the relative errors are of order $10^{-5}$, $10^{-8}$, and $10^{-15}$ respectively. All the asymptotic reflection formulas in subsequent sections have similar accuracy as that of Corollary \ref{CorollaryPochhammer}.

\section{q-gamma function} \label{SectionGamma}
The relation (\ref{qgamma}) expresses the q-gamma function in terms of the already analyzed q-Pochhammer symbol. Using this relationship, the main result of this section describing the behavior of the q-gamma function near $q=1$ is stated as the following theorem.
\begin{theorem}\label{TheoremGamma}
The asymptotic expansion of the q-gamma function \emph{(\ref{qgamma})} at $q=1$ is given by\emph{:}
\begin{equation*}
\begin{aligned}
\Gamma_q(x) = \Gamma(x) \left(\frac{\log q}{q-1} \, q^\frac{x}{4}\right)^{\!{x-1}} \prod_{k=3}^\infty e^{-\frac{\zeta(2-k)}{k!} B_k(x) (\log q)^{k-1}}.
\end{aligned}
\end{equation*}
\end{theorem}

\begin{proof}
We begin with (\ref{qgamma}) and replace the two q-Pochhammer symbols with their asymptotic expressions in Theorem \ref{TheoremPochhammer} and equation (\ref{Euler}) in the following manner.
\begin{equation*}
\begin{aligned}
\Gamma_q(x) &= (1-q)^{1-x} (q,q)_\infty (q^x,q)_\infty^{-1} \\
&= (1-q)^{1-x} \sqrt{\frac{2 \pi}{\log(1/q)}} \, e^\frac{\pi^2}{6 \log q} q^{-\frac{1}{24}} \\
&\hspace{0.16 in} \times\Bigg[\frac{\sqrt{2 \pi}}{\Gamma(x)} \left(\!\log \frac{1}{q}\right)^{\!\!\frac{1}{2}-x} \prod_{\substack{k=0, \\ k \neq 1}}^\infty e^{\frac{\zeta(2-k)}{k!} B_k(x) (\log q)^{k-1}}\Bigg]^{-1} \\
&= \Gamma(x) \left(\frac{\log q}{q-1}\right)^{\!\!x-1} \! e^\frac{\pi^2}{6 \log q} q^{-\frac{1}{24}} \\
&\hspace{0.16 in} \times e^{-\frac{\zeta(2)}{\log q} B_0(x)}  e^{-\frac{\zeta(0)}{2} B_2(x) \log q} \prod_{k=3}^\infty e^{-\frac{\zeta(2-k)}{k!} B_k(x) (\log q)^{k-1}} \\
&= \Gamma(x) \left(\frac{\log q}{q-1}\right)^{\!\!x-1} \! e^\frac{\pi^2}{6 \log q} q^{-\frac{1}{24}} \\
&\hspace{0.16 in} \times e^{-\frac{\pi^2}{6 \log q}}  q^{\frac{1}{4} (\frac{1}{6}-x+x^2)} \prod_{k=3}^\infty e^{-\frac{\zeta(2-k)}{k!} B_k(x) (\log q)^{k-1}} \\
&= \Gamma(x) \left(\frac{\log q}{q-1} \, q^\frac{x}{4}\right)^{\!{x-1}} \prod_{k=3}^\infty e^{-\frac{\zeta(2-k)}{k!} B_k(x) (\log q)^{k-1}}.
\end{aligned}
\end{equation*}
\end{proof}

The q-analog of the reflection formula for the gamma function follows from the theorem above and is stated as the following corollary.
\begin{corollary} \label{CorollaryGamma}
The q-gamma function \emph{(\ref{qgamma})} satisfies the asymptotic reflection formula\emph{:}
\begin{equation*}
\Gamma_q(x) \, \Gamma_q(1-x) \simeq \frac{\pi}{\sin(\pi x)} \, \frac{q-1}{\log q} \, q^\frac{x (x-1)}{2}.
\end{equation*}
\end{corollary}

\begin{proof}
\begin{equation*}
\begin{aligned}
\Gamma_q(x) \, \Gamma_q(1-x) &= \Gamma(x) \, \Gamma(1-x)\! \left(\frac{\log q}{q-1} \, q^\frac{x}{4}\right)^{\!{x-1}}\left(\frac{\log q}{q-1} \, q^\frac{1-x}{4}\right)^{\!-x}\\
&\text{\hspace{0.235 in}} \times \prod_{k=3}^\infty e^{-\frac{\zeta(2-k)}{k!} B_k(x) (\log q)^{k-1}}\prod_{k=3}^\infty e^{-\frac{\zeta(2-k)}{k!} B_k(1-x) (\log q)^{k-1}}\\
&= \frac{\pi}{\sin(\pi x)} \, \frac{q-1}{\log q} \, q^\frac{x (x-1)}{2} \prod_{k=3}^\infty e^{-\frac{\zeta(2-k)}{k!} [B_k(x)+B_k(1-x)] (\log q)^{k-1}}\\
&\simeq \frac{\pi}{\sin(\pi x)} \, \frac{q-1}{\log q} \, q^\frac{x (x-1)}{2}.
\end{aligned}
\end{equation*}
The product from $k=3$ to $\infty$ in the proof above equates to 1 for the same reasons as discussed in the proof of the q-Pochhammer reflection formula (Corollary \ref{CorollaryPochhammer}).
\end{proof}

As a special case of the reflection formula in Corollary \ref{CorollaryGamma}, we obtain the following proposition.
\begin{proposition} \label{PropositionGamma}
The q-gamma function at $x=1/2$ is given by the asymptotic relation\emph{:}
\begin{equation*}
\Gamma_q\left(\frac{1}{2}\right) \simeq \sqrt{\frac{\pi}{q^{1/8}} \frac{q-1}{\log q}}.
\end{equation*}
\end{proposition}

Although in this paper we assume $q<1$, it can be easily seen that the asymptotic expansion of the q-gamma function in Theorem \ref{TheoremGamma} extends to $q\ge 1$. The q-gamma expansion satisfies the $q$ inversion symmetry \cite{Jackson}:
\begin{equation}
\Gamma_q(x) = q^{\frac{(x-1)(x-2)}{2}} \Gamma_{1/q}(x).
\end{equation}
Similarly, the asymptotic expansions for the q-digamma and q-polygamma functions that follow in Sections \ref{SectionDigamma} and \ref{SectionPolygamma} are valid for $q \ge 1$ as well.

\section{q-digamma function} \label{SectionDigamma}
In Section \ref{Introduction} we related q-digamma function to the Lambert series through equation (\ref{qdigamma}). Using this relationship we calculate the behavior of the q-digamma function near $q=1$ as stated in the theorem below.
\begin{theorem}\label{TheoremDigamma}
The expansion of the q-digamma function \emph{(\ref{qdigamma})} at $q=1$ may be expressed in the following manner.
\begin{enumerate}
\item With $-\log (1-q)$ included in a single term\emph{:}
\begin{equation*}
\begin{aligned}
\psi_q(x) = \psi(x) + \log\!\left(\frac{\log q}{q-1}\right) - \sum_{k=1}^\infty \frac{\zeta(1-k)}{k!} B_k(x) (\log q)^k.
\end{aligned}
\end{equation*}

\item  With $-\log (1-q)$ included in its expanded form\emph{:}
\begin{equation*}
\begin{aligned}
\psi_q(x) = \psi(x) + \sum_{k=1}^\infty \frac{\zeta(1-k)}{k!} [1-B_k(x)] (\log q)^k.
\end{aligned}
\end{equation*}
\end{enumerate}
\end{theorem}
\begin{remark}
Both expansions are diverging series, but expansion (1) shows better asymptotic convergence if truncated at small $k$ values. Expansion (2) allows for easier calculations of the derivatives with respect to $\log q$.
\end{remark}
\begin{proof}
To prove (1), we replace $\LambertL_q(0,x)$ in (\ref{qdigamma}) with its expansion in Theorem \ref{TheoremLambert}. Combining $-\log (1-q)$ and $\log \log (1/q)$ into a single term completes the proof.
For the proof of (2), we replace $-\log(1-q)$ in (\ref{qdigamma}) with its expanded form as well. Collecting the terms in powers of $\log q$ yields the result in (2).
\end{proof}

A q-analog of the digamma reflection formula is stated as the following proposition.
\begin{proposition} \label{PropositionDigamma}
The q-digamma function \emph{(\ref{qdigamma})} satisfies the asymptotic reflection formula\emph{:}
\begin{equation*}
\psi_q(x) - \psi_q(1-x) \simeq -\pi \cot(\pi x) + \left(x-\frac{1}{2}\right) \log q.
\end{equation*}
\end{proposition}
\begin{proof}
The proof follows by taking the logarithmic derivative with respect to $x$ of the q-gamma reflection formula in Corollary \ref{CorollaryGamma}.
\end{proof}

\section{q-polygamma functions} \label{SectionPolygamma}
The q-polygamma functions are related to the Lambert series through equation (\ref{qpolygamma}). Using this relationship we state the asymptotic behavior of the q-polygamma functions in this theorem below.
\begin{theorem}\label{TheoremPolygamma}
The q-polygamma functions $\psi_q^{(m)}(x)$ for $m \ge 1$ have the following expansion at $q=1$\emph{:}
\begin{equation*}
\begin{aligned}
\psi_q^{(m)}(x) = \psi^{(m)}(x) - \sum_{k=0}^\infty \frac{\zeta(1-m-k)}{k!} B_k(x) (\log q)^{m+k}.\end{aligned}
\end{equation*}
\end{theorem}

\begin{proof}
Substituting the asymptotic expansion of $\LambertL_q(m,x)$ into (\ref{qpolygamma}) and noting that $(-1)^{1+m} \Gamma(1+m) \, \zeta(1+m,x) =\psi^{(m)}(x)$ when $m$ is a positive integer completes the proof.
\end{proof}

The q-analogs of the polygamma reflection formulas are stated as the following proposition.
\begin{proposition}\label{PropositionPolygamma}
The q-polygamma functions \emph{(\ref{qpolygamma})} asymptotically satisfy the following reflection formula.
\begin{enumerate}
\item For $m=1$\emph{:}
\begin{equation*}
\begin{aligned}
\psi_q^{(1)}(x) +\psi_q^{(1)}(1-x) \simeq [\pi \csc(\pi x)]^2 + \log q.
\end{aligned}
\end{equation*}

\item For $m \geq 2$\emph{:}
\begin{equation*}
\begin{aligned}
\psi_q^{(m)}(x) - (-1)^m \psi_q^{(m)}(1-x) \simeq \psi^{(m)}(x) - (-1)^m \psi^{(m)}(1-x).
\end{aligned}
\end{equation*}

\end{enumerate}
\end{proposition}

\begin{proof}
The proof follows from taking the $m^\text{th}$ derivative with respect to $x$ of the q-digamma reflection formula in Proposition \ref{PropositionDigamma}. In (2) we have chosen to write $\psi^{(m)}(x) - (-1)^m \psi^{(m)}(1-x)$ instead of $\frac{d^m}{dx^m} [-\pi \cot(\pi x)]$.
\end{proof}

Some simple special cases of Theorem \ref{TheoremPolygamma} and Proposition \ref{PropositionPolygamma} are expressed in the following example.
\begin{example} \label{ExamplePolygamma}
\emph{For odd $m\ge 1$ the q-polygamma functions at $x=1$ and at $x=1/2$ are given by the following asymptotic relations}:
\begin{equation*}
\begin{aligned}
&\hspace{-0.26in}(1)\hspace{0.2in} \psi_q^{(m)}(1) \simeq \psi^{(m)}(1) - \zeta(1-m) (\log q)^m - \frac{1}{2} \zeta(-m)(\log q)^{m+1}. \\
&\hspace{-0.26in}(2)\hspace{0.2in} \psi_q^{(1)}\!\!\left(\tfrac{1}{2}\right) \simeq \frac{\pi^2}{2} + \frac{1}{2} \log q.\\
&\hspace{-0.26in}(3)\hspace{0.2in} \psi_q^{(m)}\!\!\left(\tfrac{1}{2}\right) \simeq \psi^{(m)}\!\!\left(\tfrac{1}{2}\right).
\end{aligned}
\end{equation*}
\end{example}
Note that the generating function of the divisor function in (1) of Corollary \ref{CorollaryLambert} is equal to $\psi_q^{(m)}(1)/(\log q)^{m+1}$ in the example above.

\section{Jacobi theta functions}
In order to analyze the behavior of the Jacobi theta functions near $q=1$ we exploit the fact that these functions can be expressed in terms of the q-Pochhammer symbol. To simplify the proof of the main theorem that follows, we make a change of variables and rewrite the relation (\ref{triple product}) in the following manner:
\begin{equation}\label{LemmaTheta}
\vartheta_4\left(\frac{y \log q}{2 i},\sqrt{q}\right) = \left(q,q\right)_\infty \, \left[\left(q^{\frac{1}{2}+y},q\right)_\infty \left(q^{\frac{1}{2}-y},q\right)_\infty\right].
\end{equation}
Writing (\ref{triple product}) in this manner allows us to relate the asymptotic form of the fourth Jacobi theta function to two already established simpler asymptotic results: the Euler function expansion (\ref{Euler}), and the q-Pochhammer reflection formula in Corollary \ref{CorollaryPochhammer} (with  $x=\tfrac{1}{2}+y$).

We now present the main result of this section describing the asymptotic behavior of the Jacobi theta functions.
\begin{theorem}\label{TheoremTheta}
The Jacobi theta functions \emph{(\ref{theta 1})} and \emph{(\ref{theta 3})} satisfy the following asymptotic relations at $q=1$\emph{:}
\begin{equation*}
\begin{aligned}
\vartheta_1\!\left(z+\frac{\pi}{2},q\right) = \vartheta_2(z,q) \simeq \sqrt{\frac{\pi}{\log(1/q)}} \left[e^\frac{z_\pi^2}{\log q} - e^\frac{(\pi-z_\pi)^2}{\log q}\right] (-1)^{\lfloor\frac{z}{\pi}\rfloor}
\end{aligned}
\end{equation*}
and
\begin{equation*}
\begin{aligned}
\vartheta_3(z,q) = \vartheta_4\!\left(z+\frac{\pi}{2},q\right) \simeq \sqrt{\frac{\pi}{\log(1/q)}} \left[e^\frac{z_\pi^2}{\log q} + e^\frac{(\pi-z_\pi)^2}{\log q}\right],
\end{aligned}
\end{equation*}
where $z_\pi$ is $z$ mod $\pi$, and $\lfloor\frac{z}{\pi}\rfloor$ is the floor of $\frac{z}{\pi}$.
\end{theorem}

\begin{proof}
Using the asymptotic forms of the Euler function (\ref{Euler}) and Corollary \ref{CorollaryPochhammer} with $x=\tfrac{1}{2}+y$ and substituting into (\ref{LemmaTheta}) gives:
\begin{equation*}
\begin{aligned}
\vartheta_4\left(\frac{y \log q}{2 i},\sqrt{q}\right) &\simeq \sqrt{\frac{2\pi}{\log(1/q)}}\,e^{\frac{\pi^2}{6\log q}} \, q^{-\frac{1}{24}} \left[ 2 \cos(\pi y)\, e^{\frac{\pi^2}{3\log q}} q^{\frac{1}{2}\left(\!\frac{1}{12}-y^2\right)}\right]\\
&=2 \cos(\pi y) \sqrt{\frac{2\pi}{\log(1/q)}}\, e^{\frac{\pi^2}{2\log q}} \, q^{-\frac{y^2}{2}}.
\end{aligned}
\end{equation*}
By substituting $y=\frac{i2z}{\log q}$ and replacing $q$ with $q^2$ we have:
\begin{equation*}
\begin{aligned}
\vartheta_4\left(z,q \right) &\simeq 2 \cos \!\left( \frac{i\pi z}{\log q}\right) \sqrt{\frac{\pi}{\log(1/q)}}\, e^{\frac{\pi^2}{4\log q}} \, q^{-\left(\!\frac{i z}{\,\log q}\right)^2}\\
&= 2 \cosh \! \left( \frac{\pi z}{\log q}\right)\sqrt{\frac{\pi}{\log(1/q)}} \, e^{\frac{(\pi\!/2)^2+z^2}{\log q}}\\
&= \sqrt{\frac{\pi}{\log(1/q)}} \, \left[e^{\frac{(\frac{\pi}{2}+z)^2\!}{\log q}} + e^{\frac{(\frac{\pi}{2}-z)^2\!}{\log q}} \right].
\end{aligned}
\end{equation*}
Evaluating the expression above at $z+\pi/2$ and then replacing $z$ with $z\!\!\mod \pi$ completes the proof for $\vartheta_3(z,q) = \vartheta_4\!\left(z+\frac{\pi}{2},q\right)$. The modulo operation is arbitrary. However, this choice ($0 \le z < \pi$) is the most natural one for the final expression above given that both $\vartheta_3(z,q)$ and $\vartheta_4(z,q)$ are $\pi$-periodic in $z$.

The proof of the other half of the theorem follows trivially by using the final expression above and applying it in (\ref{theta1-theta4}) to calculate $\vartheta_1\!\left(z+\frac{\pi}{2},q\right) = \vartheta_2(z,q)$. The floor function factor $\lfloor\frac{z}{\pi}\rfloor$ is introduced arbitrarily at the end along with the replacement of $z$ with $z\!\!\mod \pi$ (for similar reasons as above) by noting that both $\vartheta_1(z,q)$ and $\vartheta_2(z,q)$ are $\pi$-antiperiodic in $z$.
\end{proof}

Many infinite series can be expressed in terms of the Jacobi theta functions \cite{Whittaker}. Of particular interest are the logarithmic derivatives of the Jacobi theta functions that have simple Lambert series like forms. For example,
\begin{equation}
\frac{\vartheta_1'\!\left(z ,q\right)}{\vartheta_1\!\left(z,q\right)} = \cot z + 4 \sum_{n=1}^\infty \frac{q^{2n}}{1-q^{2n}} \sin (2nz).
\end{equation}

Using Theorem \ref{TheoremTheta} we can sum such series in asymptotic form as stated below.
\begin{corollary} \label{CorollaryTheta}
The logarithmic derivatives with respect to $z$ of the Jacobi theta functions satisfy the following asymptotic relations:
\begin{equation*}
\frac{\vartheta_1'\!\left(z+\frac{\pi}{2},q\right)}{\vartheta_1\!\left(z+\frac{\pi}{2},q\right)} = \frac{\vartheta_2'(z,q)}{\vartheta_2(z,q)} \simeq \frac{1}{\log q} \left[2 \left(\!z_\pi - \frac{\pi}{2}\right) + \pi \coth\!\left(\!\frac{\pi \left(z_\pi - \frac{\pi}{2}\right)}{\log q}\!\right)\right],
\end{equation*}
and
\begin{equation*}
\frac{\vartheta_3'(z,q)}{\vartheta_3(z,q)} = \frac{\vartheta_4'\!\left(z+\frac{\pi}{2},q\right)}{\vartheta_4\!\left(z+\frac{\pi}{2},q\right)} \simeq \frac{1}{\log q} \left[2 \left(\!z_\pi - \frac{\pi}{2}\right) + \pi \tanh\!\left(\!\frac{\pi \left(z_\pi - \frac{\pi}{2}\right)}{\log q}\!\right)\right].
\end{equation*}
\end{corollary}
\begin{proof}
All the logarithmic derivatives of the theta functions are $\pi$-periodic in $z$\cite{Whittaker}. Taking the logarithmic derivatives of expressions in Theorem \ref{TheoremTheta} in the range $0\le z < \pi$, and then replacing $z$ with $z \! \!\mod \pi$ completes the proof.
\end{proof}
Finally, from applications of Theorem \ref{TheoremEisenstein} and Theorem \ref{TheoremTheta}, the Jacobi theta functions and the modified Eisenstein series can be related to each other as follows.
\begin{proposition} \label{PropositionTheta}
For all positive integers $k\ge 1$, the asymptotic expansion of modified Eisenstein series may be expressed in terms of the Jacobi theta functions as\emph{:}
\begin{equation*}
\tilde{E}_{2k}(q) \simeq \left[\frac{\vartheta_2(0,q^{1/2})+\vartheta_3(0,q^{1/2})}{2\sqrt{i}}\right]^{4k}.
\end{equation*}
\end{proposition}
\bibliographystyle{amsplain}

\end{document}